\documentclass[12pt]{amsart}
\usepackage{amsthm,amsmath,amssymb,amscd,graphics,enumerate, stmaryrd,xspace,verbatim, epic, eepic,color,url}

\usepackage{eurosym}
\usepackage[normalem]{ulem}
\usepackage[active]{srcltx}
\usepackage[all]{xypic}
\SelectTips{cm}{}

{
   \newtheorem{theorem}[subsubsection]{Theorem}
      \newtheorem*{theorem*}{Theorem}
   \newtheorem{proposition}[subsubsection]{Proposition} 
        
   \newtheorem{lemma}[subsubsection]{Lemma}

   \newtheorem{corollary}[subsubsection]{Corollary}
   
   \newtheorem*{conjecture*}{Conjecture}
   
}
{\theoremstyle{definition}
          \newtheorem*{exercise*}{Exercise}

   \newtheorem*{example*}{Example}
   \newtheorem{definition}[subsubsection]{Definition}
   \newtheorem{defn}[subsubsection]{Definition}
   \newtheorem*{definition*}{Definition}

   \newtheorem{conv}[subsubsection]{Convention}
   \newtheorem*{conv*}{Convention}
   \newtheorem{notation}[subsubsection]{Notation}
}
\newcommand{\PP}{{\mathbb{P}}}

\newcommand{\GG}{{\mathbb{G}}}

\renewcommand{\AA}{{\mathbb{A}}}

\newcommand{\tDelta}{{\widetilde{\Delta}}}

\newcommand{\fT}{{\mathfrak{T}}}

\newcommand{\cA}{{\mathcal A}}

\renewcommand{\cD}{{\mathcal D}}

\newcommand{\cO}{{\mathcal O}}

\newcommand{\cT}{{\mathcal T}}

\newcommand{\cW}{{\mathcal W}}

\def\<{\langle}
\def\>{\rangle}

\newcommand{\Spec}{\operatorname{Spec}}

\newcommand{\sh}{{\operatorname{sh}}}

\newcommand{\ocM}{\overline{{\mathcal M}}}

\newcommand{\smooth}{{\operatorname{sm}}}
\newcommand{\sing}{{\operatorname{sing}}}
\newcommand{\spe}{{\operatorname{sp}}}

\newcommand{\pd}{{\operatorname{pd}}}
\newcommand{\nd}{{\operatorname{nd}}}

\newcommand{\double}{\genfrac..{0pt}1
{\raise -1pt\hbox{$\scriptstyle\longrightarrow$}}{\raise 3pt\hbox
{$\scriptstyle\longrightarrow$}}} 

\renewcommand{\setminus}{\smallsetminus}

\def\tototi{\mathbin{\mathop{\otimes}\limits^{\raise-1pt\hbox
{$\scriptscriptstyle {\rm L}$}}}}

\def\indlim{\mathop{\vrule width0pt height7pt depth
4pt\smash{\lim\limits_{\raise 1pt\hbox to 14.5pt
{\rightarrowfill}}}}}
\def\projlim{\mathop{\vrule width0pt height7pt depth
4pt\smash{\lim\limits_{\raise 1pt\hbox to 14.5pt
{\leftarrowfill}}}}}

\newcommand\displaceamount{3pt}

\newcommand{\doubledown}{\ar@<\displaceamount>[d]\ar@<-\displaceamount>[d]}

\newcommand{\doubleup}{\ar@<\displaceamount>[u]\ar@<-\displaceamount>[u]}

\newcommand{\doubleright}{\ar@<\displaceamount>[r]\ar@<-\displaceamount>[r]}

\newcommand{\eps}{\varepsilon}

\newcommand{\fK}{\mathfrak K}

\begin{document}

\title[Configurations and compactness]{Configurations of points on degenerate varieties and properness of moduli spaces}

\author[D. Abramovich]{Dan Abramovich}
\thanks{Research of D.A. partially supported by NSF grant DMS-0335501 and DMS-0603284}  
\address{Department of Mathematics, Box 1917, Brown University,
Providence, RI, 02912, U.S.A} 
\email{abrmovic@math.brown.edu}
\author[B. Fantechi ({\tiny\today})]{Barbara Fantechi}
\thanks{Research activities of B.F. partially supported by Italian research grant PRIN 2006 ``Geometria delle variet\`a proiettive", ESF Research Network MISGAM, Marie Curie RTN ENIGMA, and GNSAGA}  
\address{SISSA,
Via Bonomea 265,
34136 Trieste, Italy} 
\email{fantechi@sissa.it}

\date{\today}


\maketitle
\setcounter{tocdepth}{1}

\tableofcontents

 \section*{Introduction} 
 
Consider a smooth variety $X$ and a smooth divisor $D\subset X$.  Kim and Sato in \cite[1.1, esp. Theorem 1]{Kim-Sato} define 
a natural compactification of $(X\setminus D)^n$, denoted $X_D^{[n]}$, which is a
moduli space of stable configurations of $n$ points lying on expansions of $(X,D)$ in the sense of \cite{Li1}. 

The purpose of this note is to generalize \cite[Theorem 1]{Kim-Sato}  to the case where $X$ is an algebraic stack; and to construct  an analogous projective moduli space $W_\pi^{[n]}$ for a degeneration $\pi:W \to B$. We construct $X^n_D$ and $W_\pi^{[n]}$  and prove their properness using a universal construction introduced in \cite{ACFW}.
We then use these spaces for a concrete application, as explained in the next paragraph.

In \cite{AF}, a degeneration formula for Gromov--Witten invariants of schemes and stacks is developed, generalizing the approach of Jun Li \cite{Li1,Li2}. This in particular requires proving properness of Li's stack of pre-deformable stable maps  in the case where the target $(X,D)$ or $W\to B$ is a  Deligne--Mumford stack. One could simply adapt Li's proof, or follow the age-old tradition of imposing such endeavor as an exercise on ``the interested reader". 

Instead, we prefer to provide a different proof here, which  uses the properness of $X_D^{[n]}$ and $W_\pi^{[n]}$. Similar ideas are used in \cite{KKO} to prove the properness of their space of ramified maps.

\begin{conv*}  To keep ideas simple we work over an algebraically closed base field $K$; a scheme will be a scheme of finite type over $K$; a curve will be a purely one-dimensional scheme; we will use the word algebraic stack in the sense of Artin, and assume locally finite type over $K$; we will write DM for Deligne-Mumford stacks. A point in an algebraic stack, and in particular in a scheme, will be a $K$-valued point.
\end{conv*}

\paragraph{\bf Achnowledgements.} Thanks to Bumsig Kim, Andrew Kresch, and Jonathan Wise, whose shared ideas and comments are manifest in this note.

\section{Stable expanded configurations}

Suppose $X$ is a smooth algebraic stack. Then $X^n$ is a parameter space for $n$ ordered points in $X$. If $D\subset X$ is a smooth divisor, we may wish to consider a space of $n$ ordered points on $X$ where points are not allowed to land in $D$, but rather $X$ is replaced by an expansion. We construct in Section \ref{confpairs} such a compactification, denoted $X^{[n]}_D$: it is an immediate generalization of the first construction in \cite{Kim-Sato} to algebraic stacks, though our method is different. 

Similarly, suppose {$\pi:W\to B$ is a flat morphism with $W$ a smooth algebraic stack, $B$ a smooth curve and all fibers $W_b$ smooth except for $W_0:=W_{b_0}$ which is the union of two smooth stacks $X_1$ and $X_2$ intersecting transversely along a divisor $D$. }
Then the self fibered product $W^n_B$ can be viewed as a parameter space of $n$ ordered points on fibers of $\pi$. 
{We} construct in Section \ref{confdeg} a space of $n$-tuples of points on the fibers which are not allowed to land in the singular locus of $\pi$ but rather $W$ is replaced by an expansion.

\subsection{Notation for degenerations}\label{notdeg}

\begin{conv}\label{degeneration_conv} 
In this section we fix $\pi: W \to B$, a flat morphism such that $B$ is a smooth curve, $W$ is a smooth algebraic stack, and $b_0\in B$ is the unique critical value of $\pi$; we set $W_0:=\pi^{-1}(b_0)$ and assume $W_0= X_1\sqcup_D X_2$  is the union of two smooth closed substacks $X_1$ and $X_2$ intersecting transversally along $D$, a smooth divisor in each $X_i$. This implies that $W_0$ is  first-order smoothable along its singular locus $D$, i.e., $N_{D/X_1}$ is dual to $N_{D/X_2}$. 
\end{conv}

We use the notation of \cite[Section 2.3]{ACFW}. The {\em expansion of length $\ell\ge0$ of $W_0$} is \begin{equation*}
  W(\ell) := X_1 \hskip -1ex \mathop{\sqcup}\limits_{\hphantom{_-}D=D_1^-}  P_1 \mathop{\sqcup}\limits_{D_1^+=D_2^-}  \cdots   \mathop{\sqcup}\limits_{D_{\ell-1}^+=D_\ell^-}  P_\ell\mathop{\sqcup}\limits_{D_\ell^+=D}\hphantom{_-} 
  X_2
\end{equation*}
where the {\em exceptional components} $P_j$ are all isomorphic to the $\PP^1$-bundle $\PP(N_{D/X_1}\oplus \cO_D)=\PP(\cO_D\oplus N_{D/X_2})$ and $D_j^{-}$, $D_j^{+}$ are 
{the} zero and infinity sections {in $P_j$} (see \cite[Definition 2.3.1]{ACFW} for details). {The automorphism group of $W(\ell)$ is defined to be $\GG_m^\ell$ where the $j$-th factor acts on the fibers of $P_j$ fixing $D_j^-$ and $D_j^+$.}

Let $\cA$ be the stack quotient $ [\AA^1/\GG_m]$, so that morphisms $S\to \cA$ are{ pairs $(L,s)$ consisting of a line bundle and section} on $S$, see e.g. \cite[Lemma 2.1.1]{cadman}. We denote by $\fT$ the universal stack of expansions of the degeneration $\varpi:\cA^2 \to \cA$ induced by $t=xy$, with universal expansion $(\cA^2)'\to \cA^2$. 

The {\em moduli stack of expansions} of $W\to B$ is $\fT_B = \fT\times_\cA B$ where $B \to \cA$ is the morphism associated to the Cartier divisor $\{b_0\} \subset B$; the universal expansion is $W' = (\cA^2)' \times_{\cA^2} W\to W$, where $W\to\cA^2$ is the smooth morphism induced by the divisors $X_1$ and $X_2$: see \cite[Definition 2.3.6]{ACFW} for details.

\subsection{Configurations on degenerations}\label{confdeg}


\begin{definition}
A {\em stable expanded configuration} $(\cW, \sigma_i)$ of degree $n$ on $W\to  B$ consists of
\begin{enumerate}
\item
a point $\cW$ of $\fT_B$, hence {either a smooth fiber $W_b$ or an expansion of $W_0$}, and
\item an ordered  collection of  $n$ smooth points $\sigma_i\in\cW^\smooth$,
\end{enumerate}
such that, in case $\cW$ is an expansion, the following stability condition holds:
\begin{itemize}
\item each exceptional component $P_j$ contains at least one $\sigma_{i}$.
\end{itemize}

An {\em isomorphism} $\rho: (\cW,\sigma_i) \to (\cW',\sigma_i')$ is an isomorphism of expanded degenerations $\rho:\cW \to \cW'$ such that $\rho\circ \sigma_i = \sigma_i'$.{ Note that the stability condition is equivalent to requiring that the only automorphism of $\cW$ fixing all the $\sigma_i$'s is the identity automorphism.}

\end{definition}

\begin{defn}
A {\em family of stable expanded configurations} of degree $n$ over a $B$-scheme $S$ is given by
\begin{enumerate}
\item a family $\cW_S\to W\times_B S$ of expanded degenerations over $S$, i.e. an object of $\fT_B(S)$;
\item $n$ sections $\sigma_i:S\to \cW_S$;
\end{enumerate}
such that for every $s\in S$ the fiber $(\cW_s,\sigma_{i,s})$ over $s$ is a stable expanded configuration of degree $n$.
\end{defn}

Morphisms of families of stable expanded configurations are defined in the obvious way; the resulting fibered category $W_{\pi}^{[n]}$ is clearly a stack, and indeed a sheaf if $W$ is a scheme or algebraic space. Composing $\sigma_i$ with $\cW \to W$, we obtain an object of the fibered power $W^n_\pi$ of $W$ over $B$, giving a morphism $W_{\pi}^{[n]} \to W^n_\pi$.

The notation should not be confused with the notation for the degree-$n$ Hilbert scheme of a surface.

The following lemma is an immediate consequence of the definition:
\begin{lemma}\label{Lem:config-open} Assume $V\subset W$ is open and $\varpi$ is the restriction of $\pi$. Then $V^{[n]}_\varpi = V^n_\varpi \times_{W^n_\pi} W^{[n]}_\pi$. 
\end{lemma}

We construct $W_{\pi}^{[n]}$ starting with a special case, described in the following lemma:

\begin{lemma}\label{Lem:basic-deg} Assume $B= \AA^1$ and $W=W_{\PP}\subset \PP^2 \times \AA^1$ is the pencil $tZ^2 = XY$. Then for every $n\ge 1$ the  stack $W_{\pi}^{[n]}$ is naturally isomorphic to the moduli space of stable weighted $n$-pointed genus 0 maps $\ocM_{0, (\epsilon,\ldots,\epsilon)}(W,\beta_F)$ in the sense of Hassett, where $\beta_F$ is the class of the fiber and $\epsilon$ is a real number in $(0,1/n)$.

In particular in this case $W_{\pi}^{[n]}$ is a smooth variety, projective over $B$ and hence over $W^n_\pi$.
\end{lemma}
\begin{proof}Fix an integer $n>0$, an $\epsilon$ as in the statement, and write $M:=\ocM_{0, (\epsilon,\ldots,\epsilon)}(W,\beta_F)$ for brevity. 

Let $S$ be a scheme, and $(\cW,\sigma_i)$ a family of stable expanded configurations over $S$. Then every fiber $\cW_s$ of $\cW$ is a nodal curve of arithmetic genus zero; moreover, if a component $\bar C$ of $\cW_s$ is contracted by the natural morphism $f:\cW\to W$, then $\bar C$ is rational and contains at least two nodes and a marked point, hence is stable since $1+1+\epsilon>2$.
This defines a natural map $W_\pi^{[n]}\to M$ as sheaves.

To define the inverse, let $S$ be a scheme and $(C_S,\sigma_{i,S},f_S)\in M(S)$; we want to show that it is a stable expanded configuration. Let $(C,x_i,f)$ be a point in $M$; since the global genus is zero, every irreducible component of $C$ is smooth and rational. Moreover, no contracted component can be a tail, i.e., intersect the rest of the curve in only one point, since if such a component contains $r$ marked points we have $1+r\eps\le 1+n\eps<2$, contradicting stability. 

Let $t\in B$ such that $f(C)=W_t$. If $t\ne 0$, $W_t$ is integral, hence there is a unique component $C_0$ of $C$ mapping isomorphically to $W_t$. If $C$ had other components, than at least one of them must be a tail since every tree with at least two vertices has at least two end vertices; since all other components are contracted, this yields a contradiction. Therefore $C\to W_t$ must be an isomorphism.

If $t=0$, then $W_t$ is the union of two irreducible components $X_1$ and $X_2$ meeting in one point $p$, and $C$ has two non-contracted components $C_1$ and $C_2$ such that $f$ maps $C_i$ to $X_i$ isomorphically. Any other irreducible component must be contracted, hence $C_1$ and $C_2$ are the only tails of $C$, which implies that $C$ is a chain of rational curves, with $C_1$ and $C_2$ as extremes. Moreover, the stability condition implies that each irreducible component must contain at least one marked point $x_i$. This concludes the proof, since other conditions (flatness of $C_S\to S$, properness of $W_S\to W\times_TS$) are part of the definition of stable maps.

{It follows that} $W_{\pi}^{[n]}$ is a smooth variety, projective over $B$ and hence over $W_{\pi}^n$ (see \cite[Theorem 1.9]{AG}, \cite[Theorem 1.1.4]{BayerManin}, \cite[Theorem 1.7]{MustataMustata}, see also \cite[Theorem 1]{Hassett}).
\end{proof}
We now return to the general case.

\begin{proposition}\label{Prop:conf-deg}
The stack $W_{\pi}^{[n]}$ is smooth and algebraic, and the morphism $W_{\pi}^{[n]} \to W^n_\pi$ is projective. In particular, if $W$ is a scheme, or a Deligne--Mumford stack, so is $W_{\pi}^{[n]}$.
\end{proposition}

\begin{proof}
{\sc Case 1:} $B= \AA^1$ and $W= W_\AA =\AA^2$ mapping via $t=xy$.  This case is the restriction of the case treated in  Lemma \ref{Lem:basic-deg} to an open set  $W_\AA \subset W_{\PP}$, so it follows by Lemma \ref{Lem:config-open}.

Note that in this case there is an action of $\GG_m$ on $B$ and a compatible action of $\GG_m^2$ on $W$; if we write $(s_1,s_2)$ for a point on $\GG_m^2$, it acts on $W$ via $(x,y) \mapsto (s_1x,s_2y)$, and $s=s_1s_2\in \GG_m$ acts on $B$ via $t\mapsto st$. This induces an action of the $n$-th fiber product $\GG_m^{n+1}$ of  $\GG_m^2$ over $\GG_m$
 on $W_{\pi}^{[n]} \to W^n_\pi$ equivariant with respect to  the $\GG_m$ action on $B$.

{\sc Case 2:}  $B= [\AA^1/\GG_m]=\cA$ and $W= W_\cA= [\AA^2/\GG_m^2]=\cA^2$.  This is the quotient of the previous case by the action of $\GG_m^{n+1}$. 
Projectivity is preserved, since the relatively ample line bundle admits an equivariant structure; indeed, every line bundle on a toric variety admits such structure, see Section 3.4 of \cite{fulton1993introduction}.
For reference below we denote the  morphism described in this case $\varpi:W_\cA \to \cA$.

{\sc General case:} We use the notation in Section  \ref{notdeg}.



Consider the smooth morphisms $B\to\cA$ given by the divisor $b_0$ and $W\to W_\cA=\cA^2$ given by the divisors $X_1$ and $X_2$:
Composing any family of stable expanded configurations of degree $n$ on $W$ with these morphisms   gives a family of stable configurations in the fibers of $W_\cA$, hence we get a natural morphism   $$W_{\pi}^{[n]} \to W_{\pi}^n \times_{(W_\cA)_{\varpi}^n} (W_\cA)_{\varpi}^{[n]}.$$

We construct an inverse of this morphism as follows. Let $S$ be a $B$--scheme, and $S\to (W_\cA)_{\varpi}^{[n]}$ a morphism over $\cA$ given by a family of stable expanded configurations  $\tilde \sigma_i:S\to (\cW_\cA)'$; then to every $B$--morphism $g:=(g_1,\ldots,g_n):S\to W_\pi^n$ we can associate a family of stable expanded configurations  on $W$ by letting $\sigma_i:S\to \cW$ be the morphisms induced by $\tilde\sigma_i$ and $g_i:S\to W$.

It is easy to see that these two constructions are inverse of each other.\end{proof}

\subsection{Other degenerate cases}\label{Sec:other} The construction given in the general case of Proposition \ref{Prop:conf-deg} apply in all cases introduced in \cite[Section 2.2]{ACFW}.  
In particular one can consider the situation of having a degenerate fiber with no chosen smoothing as introduced in  \cite[2.2.2]{ACFW}: here 
$B$ is a point {$B=\{b_0\}$} and $W$ is first order smoothable, i.e. the union of two smooth components $X_1$ and $X_2$ meeting transversally along a smooth divisor $D$, with $N_{D/X_1}$ dual to $N_{D/X_2}$.
The result is the following:

\begin{proposition}\label{Prop:conf-deg-fiber}
The stack $W_{\pi}^{[n]}$ is algebraic, and the morphism $W_{\pi}^{[n]} \to W^n_\pi$ is projective. In particular, if $W$ is a scheme, or a Deligne--Mumford stack, so is $W_{\pi}^{[n]}$.
\end{proposition}
\begin{proof} We can modify the previous proof as follows. First, consider {$B\to \cA$} induced by the unique morphism $B\to 0\subset \AA^1$; the fiber product {$W_B:=W_\cA\times_\cA B$} is the transversal union of two smooth irreducible components $W_{B,i}$, each isomorphic to {$\cA$}, meeting transversally along $\GG_m$. The pairs $(X_i,D)$ define smooth morphisms $X_i\to W_{B,i}$ and the condition on the normal bundles ensures that they glue to define a smooth morphism $W\to W_B$. The rest of the proof is the same as the general case above.
\end{proof}

Note that in this case $W_{\pi}^{[n]}$ is not smooth. Our construction gives a smooth map to  $\{b_0\} \times_\cA ({W_\cA})_{\varpi}^{[n]}$; its singularities are therefore modeled on the boundary of  $\ocM_{0, (\epsilon,\ldots,\epsilon)}(W,\beta_F)$, which is a normal crossings divisor.

\subsection{Notation for pairs}
\begin{conv}\label{pairs_conv} In this section we fix a pair $(X,D)$, where $X$ is a smooth algebraic stack and $D$ a smooth divisor in $X$.
\end{conv}

The {\em expansion of length $\ell\ge0$} of $(X,D)$ is the pair $(X(\ell),D(\ell))$ with \begin{equation*}
X(\ell) := X_1 \hskip -1ex \mathop{\sqcup}\limits_{\hphantom{_-}D=D_1^-}  P_1 \mathop{\sqcup}\limits_{D_1^+=D_2^-}  \cdots   \mathop{\sqcup}\limits_{D_{\ell-1}^+=D_\ell^-}  P_\ell
\end{equation*}
and $D(\ell):=D_\ell^-\subset P_\ell$. As in \ref{degeneration_conv} the {\em exceptional components} $P_j$ are all isomorphic to $\PP(N_{D/X_1}\oplus \cO_D)$ and $D_j^{-}$, $D_j^{+}$ are their zero and infinity sections (see \cite[Definition 2.3.1]{ACFW}). Its automorphism group is again defined to be $\GG_m^\ell$ acting componentwise on the $P_j$'s.

Following the notation of \cite[Section 2.1]{ACFW} we {let again $\cA:=[\AA^1/\GG_m]$ and} denote by $\cD\subset \cA$ the smooth divisor $[0/\GG_m]$. We denote by $\cT$ the universal stack of expansions of pairs: it parametrizes expansions of $(\cA,\cD)$, with universal expansion denoted by $(\cA',\cD')$. The stack $\cT$ also parametrizes expansions of any pair $(X,D)$, with universal family $(X',D')$ where $X' = \cA'\times_\cA X$ and  $D' = \cD' \times_\cD D$, see \cite[Definition 2.1.6]{ACFW}.

\subsection{Configurations on pairs}\label{confpairs}


\begin{definition}
A {\em stable expanded configuration} $(X',D', \sigma_i)$ of degree $n$ on $(X,D)$ consists of
\begin{enumerate}
\item
a point $(X',D')$ of $\cT$, {that is an expansion of the pair $(X,D)$} and
\item an ordered  collection of  $n$ smooth points $\sigma_i\in(X')^\smooth\setminus D'$,
\end{enumerate}
such that the following stability condition holds:
\begin{itemize}
\item each exceptional component $\PP_j$ contains at least one $\sigma_{i}$.
\end{itemize}

An {\em isomorphism} $\rho: (X',D',\sigma_i) \to (\bar X',\bar D',\bar\sigma_i)$ is an isomorphism of expanded pairs  $\rho:(X',D')\to(\bar X',\bar D') $ such that $\rho\circ \sigma_i = \bar\sigma_i$.
\end{definition}

\begin{defn}
A {\em family of stable expanded configurations} of degree $n$ over a scheme $S$ is given by
\begin{enumerate}
\item {a family of expansions $(X'_S, D'_S) \to (X,D)\times S$ of the pair $(X,D)$ parametrized by $S$, that is} an object of $\cT(S)$;
\item $n$ sections $\sigma_i:S\to X'_S$;
\end{enumerate}
such that for every $s\in S$ the fiber over $s$ is a stable expanded configuration of degree $n$.
\end{defn}

Morphisms of families of stable expanded configurations are defined as before; we denote the resulting category $X_{D}^{[n]}$, and again it is a stack with a natural map to $X^n$. 

\begin{lemma}\label{Lem:basic-pairs}
Assume $X= \PP^1$ and $D=\{0\}$. Consider the moduli stack of weighted $(n+1)$-pointed stable maps $\ocM:=\ocM_{0; (\epsilon,\ldots,\epsilon,1)}(\PP^1,1)$ of degree 1, with its last evaluation map $ev:=ev_{n+1}: \ocM \to \PP^1$. Then $X_{D}^{[n]} = ev^{-1} (\{0\})$ {and hence} is a smooth projective variety, {therefore the morphism $X_D^{[n]}\to X^n$ is projective}. 
\end{lemma}

\begin{proof}
We briefly review the argument which is  very similar to that of Lemma \ref{Lem:basic-deg}. The morphism $X_{D}^{[n]}\to \ocM$ is defined by using $D$ as $(n+1)$-st section; it clearly maps to $ev^{-1}(0)$. To define the opposite morphism, note that if $(C,x_i,f)$ is a point in $\ocM$, the curve $C$ can have only two tails, one mapping isomorphically to $\PP^1$ and the other containing the $(n+1)$-st section, so it is a chain of rational curves.

As in Lemma \ref{Lem:basic-deg}, smoothness and projectivity follow from \cite[Theorem 1.9]{AG}, \cite[Theorem 1.1.4]{BayerManin}, \cite[Theorem 1.7]{MustataMustata}, \cite[Theorem 1]{Hassett}).
\end{proof}

\begin{proposition}\label{Prop:conf-pairs}
The stack $X_{D}^{[n]}$ is smooth and algebraic, and the morphism $X_{D}^{[n]} \to X^n$ is projective. In particular, if $X$ is a scheme, or a Deligne--Mumford stack, so is $X_{D}^{[n]}$.
\end{proposition}

\begin{proof}

{\sc Case 1:} $X= \AA^1$ and  $D= \{0\}$. This case is the restriction of Lemma \ref{Lem:basic-pairs} to the open subscheme $(\AA^1)^n\subset (\PP^1)^n$, and the analog of Lemma \ref{Lem:config-open} holds.

Note that in this case there is an action of $\GG_m$ on $\AA^1$ fixing $\{0\}$. This induces an action of
$\GG_m^{n}$ on $X_{D}^{[n]} \to X^n$.

{\sc Case 2:}  $X= [\AA^1/\GG_m]=\cA$ and $D= [\{0\}/\GG_m]$.  As in the proof of Proposition \ref{Prop:conf-deg}, this is the quotient of the previous case by the action of $\GG_m^{n}$.
Projectivity is preserved, since again the relatively ample line bundle admits an  equivariant structure.

This is a universal case, the resulting  configuration stacks are denoted by $\cA_{\cD}^{[n]}\to \cA^n$.

{\sc General case:} The pair $(X,D)$ defines a smooth morphism $X \to \cA$. Composing any family of stable expanded configurations of degree $n$ on $(X,D)$ with this structure morphism gives a family of stable configurations of $(\cA,\cD)$, hence we get a natural morphism   $$X_{D}^{[n]} \to X^n \times_{\cA^n} \cA_{\cD}^{[n]}.$$ This is an isomorphism exactly as in Proposition \ref{Prop:conf-deg}.
\end{proof}

\section{Properness of the stack of predeformable maps}\label{Sec:proper} 

Defining numerical invariants via virtual classes requires constructing a suitable moduli stack of DM type, proving its properness and constructing a perfect obstruction theory. 

For instance if $\pi:W\to B$ is a smooth projective morphism, with $B$ an affine scheme, then the connected components of the stack $K(W)$ of stable maps to the fibers of $\pi$ are proper over $B$ and carry a relative perfect obstruction theory; this allows to define GW invariants and prove their invariance under deformations. Properness of $K(W)\to B$ holds under the weaker assumption that $W\to B$ be flat and projective; however, in this case the natural obstruction theory is not perfect. 

Both the properness results and the construction of perfect obstruction theories are the key technical core of Jun Li's statement and proof of the degeneration formula for GW invariants in the language of algebraic geometry. In this section we want to give an alternative proof of the properness part, extending it to the case of DM stacks; we describe the set-up following the notation of \cite[Definitions 3.1.4, 3.2.9, C.1.6]{AF}, where we combine the cases of pairs and degenerations. 

\subsection{Notation and background}

We let $W\to B$ be either one of the cases of interest: a degeneration in the sense of Convention \ref{degeneration_conv}, or a first-order smoothable variety $W\to B=\{b_0\}$ as in Section \ref{Sec:other}, or the trivial map from $X$ to a point in the context of Convention \ref{pairs_conv}; the notation $K(W)$ denotes the algebraic stack of stable maps to the fibers of $W\to B$.  {The {\em special locus} $\cW^{\spe}$ of $\cW$ is either  $\cW^{\spe} = \cW^\sing$ the singular locus for a degeneration, or $\cW^{\spe} = \cW^\sing\cup D$ in the case of a pair.}

We use $T$ to denote the corresponding stack of expansions:  either the stack of expanded degenerations $\fT_B$ in the first two cases,  or the stack of expanded pairs $\cT$; we have a natural morphism $T\to B$ and a universal expansion $W' \to W\times_BT$. If $P_j$ is an exceptional components in an expansion, we call its {\em interior} the complement $P_j^o$ of {$D_j^+\cup D_j^-$}.

We denote by $\fK$ the stack of twisted stable maps to the fibers of $W' \to T$; 
while our target is allowed to be a Deligne Mumford stack, we keep the terminology lighter by always using the shorter term {\em stable maps} to refer to either stable maps in case $W$ is a scheme, or {\em twisted stable maps} as in \cite{AV} in case it is a stack.

A point of $\fK$ is a tuple $(C,\Sigma, \cW, f)$ where $\cW\to W$ is a point in $T$, and $(C,\Sigma, f)$ is a stable map to $\cW$, where $\Sigma$ is the ordered set of marked points. Such a stable map is called {\em degenerate} if there is an irreducible component $X$ of $C$ 
mapping to $\cW^\spe$; otherwise it is called {\em non-degenerate}.  A non-degenerate map is called {\em predeformable} if {all points on $C$ mapping to the divisor $D\subset \cW$ are marked, and if } near each point mapping to $\cW^\sing$ the map is given on strict henselizations  by $$(k[x,y,z_1,\ldots,z_d]/xy)^\sh \to (k[u,v]/uv)^\sh\qquad x\mapsto u^r, y\mapsto v^r$$ for some $r>0$. Here and later  {\em sh} means strict henselization and may be thought of as a choice of local analytic coordinates.

We denote by $\fK_{\nd}$ the open substack of $\fK$ whose points are non-degenerate maps, and by $\fK_\pd\subset \fK_{\nd}$ the locally closed substack of predeformable maps, with its reduced substack structure; while in \cite[Pages 541-547]{Li1} the locus $\fK_\pd$ is endowed with a  delicate schematic structure, we do not need it since all we are interested in is properness. Finally, we denote by $K\subset \fK$ the maximal open DM substack, and we let $K_{\nd}:=K\cap \fK_{\nd}$ and $K_\pd:=K\cap \fK_\pd$. 

Let $(C,\Sigma, \cW, f)$ be a point of $\fK$; we call a smooth rational irreducible component $X$ of $C$  {\em semistable}, if it maps to a fiber of an exceptional component  $P_j$ of $\cW$, with exactly one point of $X$ meeting each of $D_j^+$ and $D_j^-$, and no other marked or singular points of $C$ on $X$. 
 (The terminology in \cite[Proof of Lemma 3.2]{Li1} is {\em trivial component}.)

Semistable components can be used to characterize points of $\fK$ belonging to the Deligne--Mumford locus $K$, see \cite[Lemma 3.2]{Li1}, \cite[Definition 3.1.5]{AF}.

\begin{lemma}\label{pointsincK} A point $(C,\Sigma,\cW,f)$ of $\fK$ is not in the Deligne--Mumford locus $K$ if and only if there is an exceptional component $P_j$  such that every component of $C$ whose image meets  the interior $P_j^o=P_j\setminus (D_j^+\cup D_j^-)$  is a semistable component.
\end{lemma}
\begin{proof}
The following argument can be found in the proof of \cite[Lemma 3.2]{Li1}. A point is in $K$ if and only if no positive dimensional subgroup of the group of automorphisms of the expansion $\cW$ lifts to an automorphism of $C$. The Deligne--Mumford condition is equivalent to ensuring that for each $P_j$ there is at least one component $X$ mapping to $P_j$ to which no covering of the $\GG_m$ action lifts. It is easy to see that the only components whose image meets $P_j^o$ to which the action lifts are exactly the semistable components.
\end{proof}

\begin{notation}\label{notval}
 When we use the valuative criterion, we always take $\Delta=\Spec R$ for $R$ a discrete valuation ring, with generic point $\eta$ and closed point $s$. 
 Given a commutative diagram $$
\xymatrix{
\eta\ar[r]^{f_\eta}\ar[d] &X\ar[d]^p\\
\Delta\ar[r]_{g} & Y}
$$
we will view liftings of $g$ to $X$ as a groupoid whose objects are the two-commutative diagrams  
$$
\xymatrix{
\eta\ar[r]^{f_\eta}\ar[d] &X\ar[d]^p\\
\Delta\ar[ru]^{f}\ar[r]_{g} & Y}
$$
and whose morphisms $f\Rightarrow f'$ are those sending the two-commutative diagram of $f$ to that of $f'$. We will sometimes write ``there is a unique lifting with certain properties'' to mean ``there exists a lifting with certain properties, and it is unique up to unique isomorphism''.
\end{notation}

\subsection{Statement and proof}

The aim of this section is to prove the following theorem:

\begin{theorem}\label{propermappd} The  morphism $K_{pd}\to K(W)$ induced by $K_{pd} \subset \fK$ is proper. 
\end{theorem}

\begin{lemma}\label{predefisclosed} The  substack  of predeformable maps 
 is closed in $\fK^u_{\nd}\subset \fK^u$, the open locus of nondegenerate maps.
\end{lemma}
\begin{proof} This is a local statement in the \'etale topology; hence we can use the proof given by Jun Li, see \cite[Lemma 2.7]{Li1}.
\end{proof}


\begin{lemma}\label{Lem:chain} Let $(C,\Sigma, f:C\to\cW)\in \fK$ be a point in the closure of $\fK_\pd$; denote $(\bar C,\bar \Sigma,\bar f)$ its image in  $K(W/T)$, i.e. the stabilisation of the composition of $f$ with the structure morphism $\cW\to W$. Then any connected component $Z$ of the locus in $C$ contracted in $\bar C$ is a chain of rational curves which maps to either a marked or singular point of $\bar C$.
\end{lemma}
\begin{proof} Each contracted component must be rational and map to a fiber of $\cW$ to $W$, hence the second part of the statement is obvious. For the first we must prove that $Z$ has no an end component $X$, i.e. one having only one node and no marked point. 

We argue by contradiction: such an end component cannot be contracted in $\cW$ since $f$ is stable, so it  must map nontrivially to some fiber in an exceptional component of $\cW$; since each fiber intersects {$\cW^{\spe}$} in two points, there is a {non-marked} point $x$ on $X$ which is smooth on $C$ but maps to $\cW^\spe$.

We use again Jun Li's argument, since this is an {\'etale local computation. In case $x\in D$, then  in a neighborhood of $x$ the divisor $D$ intersects every nearby fiber. So $x$ is in the closure of the locus of marked points and hence it is marked, contradicting the assumption.  Now consider the case when $x$ maps to $\cW^\sing$.} By assumption there exist  a family $$(f_\Delta:C_\Delta\to\cW_\Delta , \Sigma_{\Delta})\in\fK(\Delta)$$ such that $(f_\eta:C_\eta\to\cW_\eta , \Sigma_{\eta})\in \fK_\pd$ and $$(f_s:C_s\to\cW_s , \Sigma_{s}) = (f:C\to\cW , \Sigma).$$ The problem is local at  $z$ so we may assume $$\cW_\Delta^\sh= \left(\Spec R[u,v,w_1,\ldots,w_m]/(uv-a)\right)^\sh$$ for some $a\in R$,
$$C_\Delta^\sh = \left(\Spec R[x]\right)^\sh,$$ and on the central fiber
$$f^*u=0, \quad f^*v = x^r.$$

Consider the homomorphism $f_\Delta^*:R[u,v,w_i]/(uv-a)\to R[x]$ and let $u_\Delta,v_\Delta$ be the images of $u$ and $v$; they must satisfy $u_\Delta v_\Delta=a$, and be equal to $(0,x^r)$ modulo $m_s$. In particular $v_\Delta$, viewed as a polynomial in $x$, has positive degree; therefore the only possibility that its product with $u_\Delta$ has degree zero is that $u_\Delta=a=0$, which means that $f_\eta$ is not predeformable.
\end{proof}



\begin{lemma}\label{propfp}
The commutative diagram $$\xymatrix{
\fK \ar[r]\ar[d] & T\ar[d]\\
K(W) \ar[r] & B
}$$
induces a proper morphism  $\fK \to K(W)\times_B   T$.
\end{lemma}
\begin{proof}
The fibered product is $\fK(W_T/T)$ where $W_T:=W\times_B T$. The structure map $\cW\to W_T$ induces the morphism from $\fK$ to the fiber product, and it is proper by  \cite[Corollary 9.1.3]{AV}.
\end{proof}

\begin{proof}[Proof of Theorem \ref{propermappd}]
To prove properness of the morphism $K_{pd}\to K(W)$ we use the valuative criterion.  We will fix from now on a commutative diagram 
\begin{equation}\label{Eq:set-up}\xymatrix{
\eta \ar[r]\ar[d] & K_{pd}\ar[d]\ar[r] & T\ar[d]\\
\Delta\ar[r] &K(W)\ar[r] & B,}
\end{equation}
and our aim will be showing that, after a base change $\tDelta\to\Delta$, there is a unique lifting of $\Delta\to K(W)$ to $\tDelta\to K_{pd}$.
We will denote by $(\bar C, \bar \Sigma, \bar f)$ the family of stable maps over $\Delta$ corresponding to $\Delta\to K(W)$ and by $(C_\eta,\Sigma_\eta,f_\eta:C_\eta\to \cW)$ the family corresponding to $\eta\to K_{pd}$.

From Lemma \ref{propfp} it follows that
 there exists an equivalence, compatible with base change,  between liftings of $\Delta\to B$ to $T$ and liftings, up to base change $\tDelta\to\Delta$, of $\Delta\to K(W)$ to $\fK$. Given a lifting $a:\tDelta \to T$, we will denote by $(C^a,\Sigma^a, f^a:C^a\to\cW)$ the family of stable maps corresponding to the lifting $\tDelta\to \fK$ induced by $a$.
 
The key to producing a lifting   of $\Delta\to B$ to $T$ is to use an auxiliary choice of a stable configuration of points as a guide.

\begin{proposition}\label{Prop:sections}
 After a base change  $\tDelta\to\Delta$, there exist a positive integer $N$ and  closed subschemes $(\bar p_1,\ldots,\bar p_N)=:\bar P\subset \bar C$ such that\begin{enumerate}
 \item the induced morphisms $\bar p_i\to \tDelta$ are isomorphisms;
\item each $\bar p_i$ is contained in the smooth, {unmarked} locus of  $\bar C$;
\item $\bar P_\eta$ lifts uniquely to $P_\eta\subset C_\eta$;
\item every irreducible component of $\bar C_s$ intersects $\bar P_s$. 
\end{enumerate}
\end{proposition}
\begin{proof} {Write $\bar C^*$ for the open dense substack obtained by deleting the nodal and marked locus of $\bar C$. The inverse image of $\bar C^*$ in $C$ is isomorphic to $\bar C^*$, so once we construct $P$ inside $\bar C^*$,  condition (3) becomes a consequence of (2). Now $\bar C^*$ is a smooth representable curve over $\Delta$.  For each irreducible component  $X_{i}\subset \bar C^*_s$ let $\bar p_{i,s}$ be a closed point and  {$U_{i}\subset \bar C^*$ an affine neighborhood such that $U_i \cap \bar C^*_s\subset X_i$. There exists an element  $h_{s} \in \Gamma(\cO_{U_{i}\cap \bar C^*_s})$} vanishing to order 1 on  $\bar p_{i,s}$. Let $h$ be a lift of $h_s$ in $\Gamma(\cO_{U_{i}})$, and $H$ its zero locus.  Then $H \to \Delta$ is quasi-finite, and replacing $\Delta$ by the localization of $H$ at $\bar p_{i,s}$ we get a section of $U_i \to \Delta$ meeting $X_{i}$. Denoting its image in $\bar C$ by $\bar p_i$ and repeating this for all components we obtain the required set of sections $P$.
} 
\end{proof}

We now replace $\Delta$ by $\tDelta$, so we may assume sections as in the proposition exist over $\Delta$.

\begin{proposition}\label{uniqstbl}
\begin{enumerate}
\item The morphism $\eta\to \cW^N$ induced by $f_\eta$ and $P_\eta\subset C_\eta$ defines a stable configuration, and hence a morphism $\eta\to W^{[N]}_\pi$ which lifts the morphism $\Delta\to W^N$ given by $\bar f$ and $\bar P$;
\item There exists a unique lifting of $\Delta\to B$ to $a_0:\Delta\to T$ such that the induced map $s\to \cW^{[N]}$  is a stable configuration.
\end{enumerate}
\end{proposition}
\begin{proof}
(1) If $X\subset C_\eta$ is an irreducible component then $X$ contracts to a point in $\bar C_\eta$ if and only if it contains no $p_i$, if and only if it is semistable. The result then follows from Lemma \ref{pointsincK}. 

(2) This follows from the fact that $W^{[N]}_\pi\to W^N_B$ is projective by Proposition \ref{Prop:conf-deg}, hence proper and representable. 
\end{proof}

\begin{proposition}\label{ndifstable}
The lifting  $(C^{a_0}_s,\Sigma^{a_0}_s,f^{a_0}_s)\in \fK$ corresponding to $a_0$ lies in $K_\pd$.
\end{proposition}
\begin{proof}
Since $K_\pd$ is closed in $K_{\nd}$ it suffices to show that  $(C^{a_0}_s,\Sigma^{a_0}_s,f^{a_0}_s)\in K_{\nd}$, and since  $K_{\nd} =  K\cap \fK_{\nd}$ it suffices to show that is in both $K$ and $\fK_{\nd}$.

{\sc Step 1}: $(C^{a_0}_s,\Sigma^{a_0}_s,f^{a_0}_s)\in \fK_{\nd}.$
 We need to show that no component of $C^{a_0}$ maps to {$\cW^\spe$}. These components come in two types: those contracted in $\bar C_s$ and those that are not. 

{\sc Step 1a}:  components contracted in $C^{a_0}_s\to \bar C_s$. By Lemma \ref{Lem:chain}, 
any connected component of the locus in $C^{a_0}_s$ contracted in $\bar C_s$ is a chain of rational curves which maps to either a marked or singular point of $\bar C_s$. But in a chain of rational curves over a point in $\bar C$ every component maps to a fiber of $\cW \to W$, so a component mapping to $\cW^\spe$  maps to a point in $\cW$, contradicting stability. 

{\sc Step 1b}:   consider a component $X_i\subset \bar C^*_s$. Then $X_i$ is canonically a component of $(C^{a_0})^*_s$  using notation as in the proof of Proposition \ref{Prop:sections}. Since $\bar p_i$ meets $X_i$ and $\bar p_i$ does not land in $\cW^\spe$, it follows that neither does $X_i$.

{\sc Step 2}: It remains to prove that $(C^{a_0}_s,\Sigma^{a_0}_s,f^{a_0}_s)\in K$. Since the map  $s\to \cW^{[N]}$  induced by $a_0$ is stable, the configuration $(p_{i,s})$ in $\cW$ is stable. So every exceptional component of $\cW$ has at least one $f^{a_0}_s(p_i)$ in its interior some component $X_i \subset \bar C^*_s$. Since these are not semistable components, we have that $(C^{a_0}_s,\Sigma^{a_0}_s,f^{a_0}_s)\in K$, as required.



\end{proof}

\begin{proposition}\label{stblifnd}
Let $a:\Delta\to T$ be any lifting of $\Delta\to B$. If  $(C^{a}_s,\Sigma^{a}_s,f^{a}_s)\in K_\pd$ then $a$ is isomorphic to $a_0$.
\end{proposition}

\begin{proof} 
By properness of $W^{[n]}_\pi$, it suffices to show that $f^{a}_s(P)$ is a stable configuration. This means that every exceptional component of $\cW$ contains a point of $f^a_s(P)$ in its interior, and that $f_s^a(P)$ is disjoint from $\cW^\spe$.

Since the inverse image of $\cW^\spe$ is contained in its nodes and marked points, and since $P$ is chosen disjoint from those, we have that $f_s^a(P)$ is disjoint from $\cW^\spe$. On the other hand, by Lemma \ref{pointsincK}, every exceptional component of $\cW$ contains in its interior the image of some component $X_i\subset \bar C^*_s$. Since $X_i$ contains a point of $P$, every exceptional component of $\cW$ contains in its interior the image of a point of $P$, as required.

\end{proof}

To conclude the proof of  Theorem \ref{propermappd} 
we use the valuative criterion of properness. 
It is enough to show that there exists a unique lifting $a$  of $\Delta$ to $T$ such that the induced point $(C^{a},\Sigma^{a}, f^{a})$ is contained  in $K_\pd$.
 Hence it is enough to prove that a lifting $a:\Delta \to T$ induces a map to $K_\pd$ if and only if it is isomorphic to $a_0$, where $a_0$ is defined in Part (2) of Proposition \ref{uniqstbl}.  This is given by Propositions \ref{stblifnd} and  \ref{ndifstable}.
\end{proof}

\begin{corollary}\label{C:pd-proper} Under the assumptions for this section, assume moreover that $W\to B$ is proper and has projective coarse moduli space. Then $K_{pd}$ is proper over $B$.
\end{corollary}
\begin{proof}
This follows since by \cite[Theorem 1.4.1]{AV} the stack $K(W)$ is proper.\end{proof}



\providecommand{\bysame}{\leavevmode\hbox to3em{\hrulefill}\thinspace}
\providecommand{\MR}{\relax\ifhmode\unskip\space\fi MR }
\providecommand{\MRhref}[2]{%
  \href{http://www.ams.org/mathscinet-getitem?mr=#1}{#2}
}
\providecommand{\href}[2]{#2}

\end{document}